\documentclass[11pt]{amsart}
\usepackage{geometry}                
\usepackage[colorlinks=true]{hyperref}
\usepackage{enumerate}
\hypersetup{urlcolor=blue, citecolor=blue}

\newtheorem{theorem}{Theorem}[section]
\newtheorem{corollary}{Corollary}[section]

\newtheorem{lemma}[theorem]{Lemma}
\newtheorem{proposition}[theorem]{Proposition}
\newtheorem{conjecture}{Conjecture}

\theoremstyle{definition}
\newtheorem{definition}[theorem]{Definition}
\newtheorem{remark}{Remark}

\newcommand{\ep}{\varepsilon}

\newcommand{\R}{\mathbb R}
\newcommand{\N}{\mathbb N}
\newcommand{\Ss}{\mathbb S}

\def\I{\mathcal{I}}
\def\raggio{\rho}

\newcommand{\m}{{\mathcal L}^n}

\newcommand{\h}{\mathcal{H}^{n-1}}
\def\dist{\mathrm {dist}}

\def\per{\mathrm {Per}}
\def\dsim{\,{\mathbin{\vcenter{\hbox{$\scriptstyle\Delta$}}\,}}}
\def\is{\displaystyle\int_{\Ss^{n-1}}}

\def\Wi{W^{1,\infty}}

\title[On a conjectured reverse Faber-Krahn ...]{On a conjectured reverse Faber-Krahn inequality for a Steklov--type Laplacian eigenvalue
}

\author[V. Ferone, C. Nitsch, C. Trombetti]{Vincenzo Ferone, Carlo Nitsch, Cristina Trombetti}
\date{}                                           
\address{\vskip.2cm\noindent Vincenzo Ferone, Carlo Nitsch, Cristina Trombetti\hfill\break\vskip-.2cm
\noindent Dipartimento di Matematica e Applicazioni ``R. Caccioppoli'', Universit\`{a}
degli Studi di Napoli ``Federico II'', Complesso Universitario Monte S. Angelo, via Cintia
- 80126 Napoli, Italy. \hfill\break\vskip-.2cm
\noindent e-mail: \tt ferone@unina.it; c.nitsch@unina.it;
cristina@unina.it}


\begin{document}
\maketitle

\begin{abstract}
For a given bounded Lipschitz set $\Omega$, we consider a Steklov--type eigenvalue problem for the Laplacian operator whose solutions provide extremal functions for the compact embedding $H^1(\Omega)\hookrightarrow L^2(\partial \Omega)$. We prove that a conjectured reverse Faber--Krahn inequality holds true at least in the class of Lipschitz sets which are ``close'' to a ball in a Hausdorff metric sense. The result implies that among sets of prescribed measure, balls are local minimizers of the embedding constant.
\end{abstract}

\section{Introduction}\label{introduction}

For any given open bounded Lipschitz set $\Omega\subset \R^n$ ($n\ge2$) the compact trace embedding $H^1(\Omega)\hookrightarrow L^2(\partial \Omega)$ allows us to define the positive quantity
\begin{equation}\label{eq_eigenvalue}
\lambda(\Omega)=\mathop{\min_{w\in H^1(\Omega)}}_{ w\ne 0}  \frac{\displaystyle\int_\Omega |Dw|^2\,dx+\int_\Omega w^2\,dx}{\displaystyle\int_{\partial \Omega}w^2\, d\h},
\end{equation}
so that any extremal function $u$ (a function achieving the minimum in \eqref{eq_eigenvalue}) is the solution to a Steklov--type eigenvalue problem
\begin{equation}\label{eigen_problem}
\left\{
\begin{array}{ll}
-\Delta u+u=0 & \mbox{in $\Omega$}\\\\
\dfrac{\partial u}{\partial \nu} = \lambda(\Omega) \,u & \mbox{on $\partial\Omega$,}
\end{array}
\right.
\end{equation}
where $\nu$ is the outer unit normal to $\partial \Omega$.
Problem \eqref{eigen_problem} has been widely investigated for instance in \cite{DPF,FBLR,FBR,FBR1,MR,R1}. The eigenvalue $\lambda(\Omega)$ is the reversed squared norm of the trace embedding operator $\mathbb{T}_\Omega:H^1(\Omega)\to L^2(\partial \Omega)$, and very often such a norm is also called ``sharp embedding constant'' drawing the attention to the fact that it is the smallest possible constant $C$ for which the Sobolev--Poincar\'e  trace inequality
\begin{equation}\label{trace_embedding}
\|w\|_{L^2(\partial\Omega)}\le C \,\|w\|_{H^1(\Omega)} 
\end{equation}
holds true for all $w\in H^1(\Omega)$.

Motivated by the study of the optimal embedding constants, and following the intuition that among sets of given volume $\lambda(\cdot)$ might be maximal on balls, J.D. Rossi proved (see \cite{R1}) that the ball is indeed a stationary set for the shape functional $\lambda(\cdot)$ under smooth volume preserving perturbation of the domain. More recently A. Henrot\footnote{Open Problem session of ``Shape optimization problems and spectral theory" Conference held at CIRM Marseille, France, June 2012. \href{http://www.lama.univ-savoie.fr/ANR-GAOS/CIRM2012/}{http://www.lama.univ-savoie.fr/ANR-GAOS/CIRM2012/}}, in analogy to the celebrated Brock and Weinstock inequality \cite{B,W} for the Steklov eigenvalue, proposed the following conjecture


\begin{conjecture}\label{conj}
For any given open bounded Lipschitz set $\Omega$, then
\[
\lambda(\Omega)\le\lambda(\Omega^\sharp),
\]
and therefore, among open bounded Lipschitz sets of given measure, the ball achieves the worst (least) embedding constant in \eqref{trace_embedding}.
\end{conjecture}

As usual by $\Omega^\sharp$ we denote the ball centered in the origin and having the same Lebesgue measure ($\m$) as $\Omega$.

Our main result is that balls are indeed local maximizers of $\lambda$ in the $L^\infty$ topology.
We prove that:
\begin{theorem}[Main Theorem]\label{mainteo} Let $\Omega\subset\R^n$ be an open bounded Lipschitz set, then
\begin{equation}\label{eq_main}
\lambda(\Omega)\le\lambda(\Omega^\sharp),
\end{equation}
provided up to translations
\begin{equation}\label{quasipalla}
\{x\in\R^n:\dist(x,\R^n\setminus\Omega^\sharp)>\delta\}\subset\Omega\subset \{x\in\R^n:\dist(x,\Omega^\sharp)<\delta\},
\end{equation}
for some positive constant $\delta$ that depends on $\m(\Omega)$ and $n$ only. 
Moreover equality holds in \eqref{eq_main}, under the constraint \eqref{quasipalla}, if and only if $\Omega$ is a ball. 
\end{theorem}

\begin{remark}
It is important to warn the reader that since $\lambda$ is not scaling invariant then $\delta$,  in Theorem \ref{mainteo}, has a genuine dependence on $\m(\Omega)$. 
\end{remark}


Isoperimetric inequalities for eigenvalues of elliptic operators (Laplacian above all) is an active field of research \cite{H,K}. Without presuming to give an exhaustive picture on the state-of-the-art, we consider just four remarkable examples (see for instance \cite{AB, Bo2, B, BucD, D1, D2, D3, Fa, H, Ka, K, Kr1, Sz, We, W} for more details):
\begin{equation}
\lambda_D(\Omega)=\mathop{\min_{w\in H_0^1(\Omega)}}_{ w\ne 0}  
\frac{\displaystyle\int_\Omega |Dw|^2\,dx}{\displaystyle\int_{\Omega}w^2\, dx},	\qquad \lambda_D(\Omega)\ge\lambda_D(\Omega^\sharp) \tag{Faber--Krahn}
\end{equation}

\begin{equation}
\lambda_N(\Omega)=\mathop{\min_{w\in H^1(\Omega)}}_{ w\ne 0,\, \int_\Omega w=0}  
\frac{\displaystyle\int_\Omega |Dw|^2\,dx}{\displaystyle\int_{\Omega}w^2\, dx}, 	\qquad \lambda_N(\Omega)\le\lambda_N(\Omega^\sharp) \tag{Szeg\"o--Weinberger}
\end{equation}

\begin{equation}
\lambda_S(\Omega)=\mathop{\min_{w\in H^1(\Omega)}}_{ w\ne 0}  
\frac{\displaystyle\int_\Omega |Dw|^2\,dx}{\inf_c\displaystyle\int_{\partial \Omega}(w-c)^2\, d\h}, 	\qquad \lambda_S(\Omega)\le\lambda_S(\Omega^\sharp) \tag{Brock--Weinstock}
\end{equation}

\begin{equation}
\lambda_R(\Omega)=\mathop{\min_{w\in H^1(\Omega)}}_{ w\ne 0}  
\frac{\displaystyle\int_\Omega |Dw|^2\,dx+\alpha\displaystyle\int_{\partial\Omega} w^2\,d\h}{\displaystyle\int_{\Omega}w^2\, dx}, 	\qquad \lambda_R(\Omega)\ge\lambda_R(\Omega^\sharp)\quad \text{if }\alpha\ge0. \tag{Bossel--Daners}
\end{equation}

Each  of the previous minimization provides the least positive eigenvalue of the Laplacian for a specific boundary condition. The subscript $D,N,S,R$ stand for
Dirichlet, Neumann, Steklov, Robin, boundary conditions. Each  of these eigenvalues is related to some embedding constant for an inequality of Sobolev--Poincar\'e type. In all the cases, among sets of given measure, balls are extremals. For more general optimal inequalities of Sobolev--Poincar\'e type we refer for instance to \cite{Maz} and, among the others, also to some recent results in \cite{AMR, Ci, CFNT, EFKNT, ENT, FNT, MV, Va}.

Interestingly enough, very little is known about the Bossel--Daners
inequality for negative $\alpha$ (see \cite{D2}). On the other hand, by trivial scaling arguments and  monotonicity of $\lambda_R(\cdot)$ with respect to $\alpha$, we will show that Conjecture \ref{conj} is equivalent to the the following one.
\begin{conjecture}\label{conj2}
For any given open bounded Lipschitz set $\Omega$ and $\alpha<0$, then
\begin{equation}\label{barek}
\lambda_R(\Omega)\le\lambda_R(\Omega^\sharp).
\end{equation}
%
\end{conjecture}

Actually, Conjecture \ref{conj2} has been addressed in 1977 by Bareket in \cite{B2}, 
where
she also provided a partial answer in two dimensions: for a given smooth simply connected set which is ``nearly circular'' there exists $\alpha<0$, with $|\alpha|$ small enough, such that \eqref{barek} holds true. To our knowledge, Conjecture \ref{conj2} is still open (see, for instance, \cite{BD}). Furthermore, very recently
{ it has been remarked that so far it is still unknown whether balls have any kind of local maximizing property (for $\lambda_R(\cdot)$ with $\alpha<0$)\footnote{This open question was raised by the Working Group on ``Low eigenvalues of Laplace operator'', during the Workshop on ``New trends in shape optimization" held at De Giorgi Center, Pisa, July 2012. \href{http://www.lama.univ-savoie.fr/ANR-GAOS/CRM\%20Pisa/index.html}{http://www.lama.univ-savoie.fr/ANR-GAOS/CRM\%20Pisa/index.html}}.} 
From this point of view Theorem \ref{mainteo} provides a positive answer to the last question, that can be summarized in the following statement.
\begin{corollary}\label{corol} Let $\Omega\subset\R^n$ be an open bounded Lipschitz set, then
\begin{equation}\label{eq_main_c}
\lambda_R(\Omega)\le\lambda_R(\Omega^\sharp),
\end{equation}
for a given $\alpha<0$, provided up to translations
\begin{equation}\label{quasipalla_c}
\{x\in\R^n:\dist(x,\R^n\setminus\Omega^\sharp)>\delta\}\subset\Omega\subset \{x\in\R^n:\dist(x,\Omega^\sharp)<\delta\},
\end{equation}
for some positive constant $\delta$ that depends on $\m(\Omega)$, $n$ and $\alpha$ only. 
Moreover equality holds in \eqref{eq_main_c}, under the constraint \eqref{quasipalla_c}, if and only if $\Omega$ is a ball. 
\end{corollary}


Now, concerning the proof of our results, the maximization of the eigenvalue $\lambda$ among Lipschitz sets 
is performed introducing a weighted isoperimetric problem which involves modified Bessel functions.
It could be of some interest to mention that inequality $\lambda_D(\Omega)\ge \lambda_D(\Omega^\sharp)$ was conjectured by Lord Rayleigh in 1877 and, as described in \cite{D3}, he provided a proof in the case of nearly circular sets (in the plane) using perturbation series involving Bessel functions.

We start by considering the function
$$z(x)=|x|^{1-\frac n2}\I_{\frac{n}{2}-1}(|x|), \qquad x\in\R^n,$$
where $\I_\nu$ denotes the modified Bessel function of order $\nu$ (see for instance \cite{AS}). When $\Omega=\Omega^\sharp$, $z$ is the extremal function in \eqref{eq_eigenvalue}, unique up to a multiplicative factor (\cite{MR}).
Actually $z$ is an analytic function in the whole $\R^n$ where it solves   
$$-\Delta z+z=0,$$
and it can be used as a test function in \eqref{eq_eigenvalue} even when $\Omega\ne\Omega^\sharp$. 
For any given bounded Lipschitz set $\Omega$ we define the notions of \emph{weighted volume} and \emph{weighted perimeter} by
\begin{equation*}
\displaystyle V(\Omega)=  \int_\Omega (|Dz|^2+z^2)\,dx,
\end{equation*}
and
\begin{equation*}
\displaystyle P(\Omega)=\int_{\partial \Omega} z^2\,d\h.
\end{equation*}

Then, by testing the right hand side of \eqref{eq_eigenvalue} with $z$, we get
$$\lambda(\Omega)\le\frac{V(\Omega)}{P(\Omega)},$$
with equality at least in the case $\Omega=\Omega^\sharp.$

This naturally suggests a way to prove Conjecture \ref{conj} by looking for the inequality 

\begin{equation}\label{key_intro}
\frac{V(\Omega)}{P(\Omega)}\le\frac{V(\Omega^\sharp)}{P(\Omega^\sharp)}.
\end{equation}
Unfortunately in general \eqref{key_intro} is false. For instance, let us consider $x_0\in\R^n\setminus\{0\}$. For what we have said before the function $z(x-x_0)$ is the unique extremal in \eqref{eq_eigenvalue} when $\Omega=B_\raggio(x_0)$ is the ball of radius $\raggio$ centered at $x_0$, and in this case we have
$$\frac{V(B_\raggio(x_0))}{P(B_\raggio(x_0))}>\lambda(B_\raggio(x_0))=\lambda(B_\raggio(0))=\frac{V(B_\raggio(0))}{P(B_\raggio(0))}.$$

We have used that $\lambda(\cdot)$ is invariant under translations, { and the resulting inequality }emphasizes that the same is not true for the ratio $\displaystyle \frac{V(\cdot)}{ P(\cdot)}$.
As a consequence, if there is any hope to prove \eqref{key_intro}, then it is crucial to carefully choose a suitable reference system for each $\Omega$.

\vskip.4cm

The paper is divided into two parts.
\vskip.4cm
{\bf Part 1.} In Section \ref{s_reduction} we give a proof of Conjecture \ref{conj}, by way of \eqref{key_intro}, for nearly spherical sets. To formulate the result we first need some definition.
\begin{definition}[$\mathcal{N}(n,\ep)$ functions] For given $n\in\N$ and $\varepsilon>0$ we denote by $\mathcal{N}(n,\varepsilon)$ the set of functions $v\in W^{1,\infty}(\Ss^{n-1})$
such that 
\begin{enumerate}[1.)]
\item $\|v\|_{\Wi}\le\ep$\\
\item $\displaystyle\frac1n\int_{\Ss^{n-1}}(1+v(\xi))^n\,d\sigma_\xi=\omega_n$ \text{(\emph{Volume constraint})}\\
\item $\displaystyle\int_{\Ss^{n-1}}(1+v(\xi))^{n+1}\xi\,d\sigma_\xi=0$ \text{(\emph{Barycenter constraint})}. 
\end{enumerate}
Here we denote by $\sigma_\xi$ the surface area measure on $\Ss^{n-1}$ and we denote as usual by $\omega_n$ the volume of the unit ball in $\R^n$.
\end{definition}

Then we need the notion of nearly spherical sets (see also \cite{Fu1}).

\begin{definition}[\emph{$(n,\omega,\ep)$--NS} sets]\label{def_nearly}
Let $\omega>0$ and $0\le \ep\le 1$. We say that an open set $\Omega\subset\R^n$ is \emph{$(n,\omega,\ep)$--NS} (that is Nearly Spherical) if 
there exists $v\in \mathcal{N}(n,\ep)$ such that, possibly up to a translation, the boundary of $\Omega$ in polar coordinates
$(r,\xi)\in[0,\infty]\times \Ss^{n-1}$, can be represented as $$r(\xi)=\raggio(1+v(\xi)).$$
Here $\raggio=\left(\omega/\omega_n\right)^{1/n}$ is the radius of the ball having same measure as $\Omega$.
\end{definition}
 
It is clear now why we labeled 2.) and 3.) by \emph{Volume constraint} and \emph{Barycenter constraint} respectively. In fact the \emph{Volume constraint} condition implies that when $\Omega$ is $(n,\omega,\ep)$--NS than $\m(\Omega)=\omega$, while the \emph{Barycenter constraint} condition establishes that in the reference frame where the boundary
of $\Omega$ is represented as $$r(\xi)=\raggio(1+v(\xi))$$ then the barycenter of $\Omega$ is placed in the origin.

The main statement of the section is the following.
\begin{theorem}\label{teonearly}
In any dimension $n$ and for every positive constant $\omega$ there exist two constants $\varepsilon$ and $K$ depending just on $n$ and $\omega$ such that, if $\Omega$ is any $(n,\omega,\ep)$--NS set and $v(\xi)\in\mathcal{N}(n,\ep)$ provides the polar representation of its boundary, then
\begin{equation*}
\begin{split}
\lambda(\Omega^\sharp)=\frac{V(\Omega^\sharp)}{P(\Omega^\sharp)}  \ge \frac{V(\Omega)}{P(\Omega)} \left(1+K(n,\omega)\int_{\Ss^{n-1}}v^2(\xi)\,d\sigma_\xi\right)\\
\ge \lambda(\Omega)\left(1+K(n,\omega)\int_{\Ss^{n-1}}v^2(\xi)\,d\sigma_\xi\right).
\end{split}
\end{equation*}
\end{theorem}

We explicitly observe that the above theorem provides an estimate on how close $\Omega$ is to  a ball in terms of the eigenvalue gap. Recent results for the first Dirichlet eigenvalue in this direction
are contained, for example, in \cite{BNT, BDe, FMP, N1}.

\vskip.4cm

\noindent{\bf Part 2.} In Section \ref{s_main} we prove Theorem \ref{mainteo}. 
Our ultimate goal is to prove that for any given $\omega>0$ there exists $\delta>0$ such that the ball centered in the origin is the unique minimizer of
\begin{equation}\label{originale}
P(\Omega)-\frac{P(\Omega^\sharp)}{V(\Omega^\sharp)}V(\Omega)
\end{equation}
among all finite perimeter sets $\Omega$, having Lebesgue measure $\m(\Omega)=\omega$, barycenter $X(\Omega)$ in the origin, and such that $B_{\raggio-\delta}\subset\Omega\subset B_{\raggio+\delta}$.
Obviously this approach first require a natural generalization of notion of $V$ and $P$ for finite perimeter sets. We end up then with a subtle weighted constrained isoperimetric problem which can not be elementarily tackled. However we replace such a constrained minimization by a penalized one, following an original powerful idea which has been introduced in \cite{CL} to prove a stability result for the classical isoperimetric inequality. Since then, such a technique has been further developed and refined in several different contexts, but among the others our construction is reminiscent of the one worked out in \cite{AFM}. 

The core of the strategy is to show that the previous minimization is indeed equivalent to minimizing the functional 
$$P(\Omega)-\frac{P(\Omega^\sharp)}{V(\Omega^\sharp)}V(\Omega)+\Lambda_1|X(\Omega)|+\Lambda_2\,|\,\m(\Omega)-\omega|+\Lambda_3(V(\Omega\backslash B_{\raggio+\delta})+V(B_{\raggio-\delta}\backslash \Omega))$$
for some choice of the positive constants $\Lambda_1$, $\Lambda_2$ and $\Lambda_3$.


For such a functional it is easy to show that for any given $\omega>0$ and $\delta>0$, every minimizer is a so called almost-minimizer for the Euclidean perimeter. We refer the interested reader to \cite{A,T1,T2}, as well as to \cite{Ma} and the references therein. For our purposes it is enough to recall that a finite perimeter set $\Omega$ is an almost-minimizer for the Euclidean perimeter $\per(\Omega)$ if there exist two positive constants $K$ and $r_0$, such that
\begin{equation}\label{almminp}
 \per(\Omega)\le  \per(\tilde\Omega)+K r^n,
\end{equation}
whenever $\Omega\dsim\tilde\Omega\subset\subset B_r(x_0)$ and $0<r<r_0$. An almost-minimizer for the perimeter $\Omega$ has reduced boundary which is a $C^{1,1/2}$ hypersurface. Moreover if $\Omega_h$ is any sequence of almost-minimizers with uniform constants $K$ and $r_0$, converging to a ball in $L^1$ as $h\to\infty$, then $\Omega_h$ is nearly spherical in the sense of Definition \ref{def_nearly} for sufficiently large $h$. Furthermore if $v_h(\xi)\in\mathcal{N}$ is the corresponding function which, according to Definition \ref{def_nearly}, provides the polar representation of $\partial\Omega$, then $v_h\to0$ in $C^{1,\alpha}(\Ss^{n-1})$ for every $\alpha\in(0,1/2)$, as $h\to \infty$.

Combining all these results, we have that for any given positive $\omega$, if $\delta>0$ is small enough, every minimizer of the constrained functional \eqref{originale} is a nearly spherical set. Therefore in view of Theorem \ref{teonearly}, possibly choosing a smaller $\delta$,  the minimizer is also unique and it is a ball.

\section{Proof of Theorem \ref{teonearly}}\label{s_reduction}

{ Throughout this section, whenever we consider a nearly spherical set, we will choose a reference frame such that the barycenter is set in the origin. Since $\lambda(\cdot)$ is invariant under translations our choice bears by no means a loss of generality.}

For the sake of simplicity we split the proof in several steps. 
Throughout the paper, $B_r=B_r(0)$, with $r>0$, and, given $\omega>0$, we denote $\raggio=\left(\omega/\omega_n\right)^{1/n}$, so that $B_\raggio$ belongs to $(n,\omega,\ep)$--NS. 

We set
\begin{equation}\label{eq_h}
h_\raggio(t)=\left((t\raggio)^{1-\frac{n}{2}}\,\I_{\frac{n}{2}-1}(t\raggio)\right)^2
\end{equation}
and
\begin{equation}\label{eq_f}
f_\raggio(t)=\frac{h_\raggio'(t)}{2\raggio}=(t\raggio)^{2-n}\I_{\frac{n}{2}-1}(t\raggio)\,\I_{\frac{n}{2}}(t\raggio),
\end{equation} 
where we have used the first one of the following derivation rules for Bessel functions
\begin{equation}\label{besseld}
\I_{\nu}'(s)=\frac{\nu}s\I_{\nu}(s)+\I_{\nu+1}(s) \qquad s\in\R,
\end{equation} 
\begin{equation}\label{besseld2}
\I'_{\nu+1}(s)=\I_\nu(s)-\frac{\nu+1} s\I_{\nu+1}(s) \qquad s\in\R.
\end{equation}

Before entering into the details we recall that (see for instance \cite{AS}) the functions 
 $s^{1-n/2}\I_{\frac{n}{2}-1}(s)$ and $s^{1-n/2}\I_{\frac{n}{2}}(s)$ are increasing and analytic in $(0,\infty)$ and
 for all $\alpha\in\R$ and $n\in\N$ we have
\begin{eqnarray}
\lim_{s\to 0}s^{1-n/2}\I_{\frac{n}{2}-1}(s)&=&  \frac1{2^{\frac{n}{2}-1}\Gamma(\frac{n}{2})}\label{bessel_1},\\
\lim_{s\to +\infty}s^{\alpha}\I_{\frac{n}{2}-1}(s)&=&+\infty.\label{bessel_2}
\end{eqnarray}

Our starting point is the following estimate.

\begin{lemma}\label{lem_stima}
If for some $n\in\N$, $\omega>0$ and $0<\ep<1$ a set $\Omega$ belongs to $(n,\omega,\ep)$--NS and $v\in\mathcal{N}(n,\ep)$ is the function that describes its boundary then
\begin{equation}\label{eq_stima}
\lambda(\Omega)\le{ \frac{V(\Omega)}{P(\Omega)}}=\frac{\is f_\raggio(1+v(\xi))\,(1+v(\xi))^{n-1}\,d\sigma_\xi}{\is h_\raggio(1+v(\xi))\,(1+v(\xi))^{n-1}\sqrt{1+\frac{|Dv(\xi)|^2}{(1+v(\xi))^2}}\,d\sigma_\xi},
\end{equation}
Moreover if $\Omega=\Omega^\sharp\equiv B_\raggio$ then equality holds in \eqref{eq_stima} and $\lambda(\Omega^\sharp)=\dfrac{f_\raggio(1)}{h_\raggio(1)}$.
\end{lemma}
\begin{proof}
The statement follows at once from the variational formulation of $\lambda(\Omega)$ as
\begin{equation}\label{eq_variational2}
\lambda(\Omega)=\min_{w\in H^1(\Omega)}\frac{\displaystyle\int_\Omega(|Dw|^2+w^2)\,dx}{\displaystyle\int_{\partial \Omega}w^2\,d\h},
\end{equation}
and the definition of $P(\Omega)$ and $V(\Omega)$ given in the introduction.

Indeed, if  $z(x)=|x|^{1-\frac{n}{2}}\I_{\frac{n}{2}-1}(|x|)$, one has
$$-\Delta z(x)+z(x)=0\qquad \forall x\in\R^n $$
and
\begin{equation}\label{eq_stima2}
\lambda(\Omega)\le{\frac{V(\Omega)}{P(\Omega)}\equiv}\frac{\displaystyle\int_\Omega(|Dz|^2+z^2)\,dx}{\displaystyle\int_{\partial \Omega}z^2\,d\h}
=\frac{\displaystyle\int_{\partial\Omega} \frac{\partial z}{\partial \nu}z\,d\h}{\displaystyle\int_{\partial \Omega} z^2\,d\h},
\end{equation}
where $\nu$ is the unit outer normal to $\partial \Omega$.

Then, to go from \eqref{eq_stima2} to \eqref{eq_stima} we just use the explicit representation of $u$ in terms of modified Bessel functions together with the explicit representation of the boundary of $\Omega$ in terms of $v(\xi)$ according to Definition \ref{def_nearly}. 

Finally, we observe that $z(x)=|x|^{1-\frac{n}{2}}\I_{\frac{n}{2}-1}(|x|)$ is precisely the first eigenfunction of \eqref{eigen_problem} when $\Omega$ is any 
ball centered in the origin. Indeed, $z(x)$ achieves the minimum on the right hand side of \eqref{eq_variational2} anytime $\Omega=\Omega^\sharp$. 
\end{proof}

The following Lemma, whose proof is included for completeness (see also \cite{Fu1}), holds.
\begin{lemma}\label{lem_expans}
Let $n\in\N$. There exists a positive constant $C$ which depends on $n$ only, such that for any given $0<\varepsilon<1$ and for all $v\in\mathcal{N}(n,\varepsilon)$ then

\begin{enumerate}[(E1)]
\item $\left| (1+v)^{n-1}-\left( 1+(n-1)v+(n-1)(n-2)\frac{v^2}{2} \right)  \right| \le C\varepsilon v^2\quad \text{on } \Ss^{n-1}.$
\item $\displaystyle 1+\frac{|Dv|^2}{2}-\sqrt{1+\frac{|Dv|^2}{(1+v)^2}}\le C\varepsilon \left(v^2+|Dv|^2\right)\quad \text{on } \Ss^{n-1}$.
\end{enumerate}
Moreover

\begin{equation}\label{eq_volume}
\left|\int_{\Ss^{n-1}}v(\xi)\,d\sigma_\xi+\frac{n-1}{2}\int_{\Ss^{n-1}}v^2(\xi)\,d\sigma_\xi\right|\le C\varepsilon\|v\|^2_{L^2}.
\end{equation}

\end{lemma}

\begin{proof}
Although the same constant $C$ appears in (E1), (E2) and \eqref{eq_volume}, throughout this proof, for the sake of simplicity, the constant $C$ is meant to be any constant that may be determined in terms of $n$ alone.

That said, inequality (E1) follows immediately from
$$(1+v)^{n-1}=\sum_{k=0}^{n-1}\binom {n-1}kv^k,$$ using the fact that $|v|\le\varepsilon$.

On the other hand, to prove (E2) we just use the trivial inequality $\sqrt{1+s}\ge 1+\frac s2-\frac {s^2}4$ which holds true for all nonnegative $s$, to get

\begin{eqnarray*}
\sqrt{1+\frac{|Dv|^2}{(1+v)^2}}&\ge& 1+\frac{|Dv|^2}{2(1+v^2)}-\frac{|Dv|^4}{4(1+v^2)^2}\\
&\ge& 1+\frac{|Dv|^2}{2(1+\varepsilon^2)}-\frac{|Dv|^4}{4}\\
&\ge& 1+\frac{|Dv|^2}{2}-C\varepsilon \left(v^2+|Dv|^2\right).
\end{eqnarray*}

Finally, since $v\in \mathcal{N}(n,\varepsilon)$ 
we know that
$$\frac1n\int_{\Ss^{n-1}}(1+v(\xi))^{n}\,d\sigma_\xi=\omega_n,$$
and integrating over $\Ss^{n-1}$ the identity
$$(1+v)^{n}=\sum_{k=0}^{n}\binom nkv^k,$$
we easily get also inequality \eqref{eq_volume}.
\end{proof}

A trivial consequence of the analyticity of $f_\raggio(\cdot)$ and $h_\raggio(\cdot)$ is the following Lemma.

\begin{lemma}\label{lem_taylor}
Let $n\in\N$ and $\omega>0$, and let $f$ and $h$ be the functions defined in \eqref{eq_h} and \eqref{eq_f}. There exists a positive constant $K$ depending on $n$ and $\omega$ alone such that, for any given $0<\varepsilon<1$ and for all $v\in\mathcal{N}(n,\varepsilon)$ then

\begin{enumerate}[(T1)]
\item $|h_\raggio(1+v)-h_\raggio(1)-h_\raggio'(1)v-h_\raggio''(1)\frac{v^2}{2}|\le K\varepsilon v^2\quad \text{on } \Ss^{n-1}$
\item $|f_\raggio(1+v)-f_\raggio(1)-f_\raggio'(1)v-f_\raggio''(1)\frac{v^2}{2}|\le K\varepsilon v^2\quad \text{on } \Ss^{n-1}$
\end{enumerate}
\end{lemma}




\begin{lemma}[Poincar\'e Inequality \cite{Fu1}]\label{lem_poinc}
Let $n\in\N$, then there exists a suitable positive constant $C$ such that for any given $0<\varepsilon<1$ and for all $v\in\mathcal{N}(n,\varepsilon)$ then
\begin{equation}\label{eq_poinc}
\|Dv\|_{L^2}^2\ge 2n (1-C\varepsilon)\|v\|_{L^2}^2.
\end{equation}
\end{lemma}

\begin{remark}
We also notice that for any given function $v(\xi)\in H^1(\Ss^{n-1})$ a well known Poincar\'e Inequality 
holds
\begin{equation}\label{eq_realpoinc}
\|Dv\|_{L^2}^2\ge (n-1)\left\|v-\bar v\right\|_{L^2}^2,
\end{equation}
where $\bar v$ is the average of $v$ over $\Ss^{n-1}$.

Unluckily inequality \eqref{eq_realpoinc} does not satisfy our needs since the constant $n-1$ is not large enough to go further in the proof of Proposition \ref{main_prop}. 
Indeed we will take advantage of the fact that we are working with functions belonging to $\mathcal N(n, \varepsilon)$ which is a proper subset of $H^1(\Ss^{n-1})$, and therefore we are able to reach the better (larger) constant in \eqref{eq_poinc}. Indeed this Lemma means that we can get as close to the embedding constant $2n$ as we wish, provided $\ep$ is chosen small enough. 
\end{remark}

\begin{proof}[Proof of Lemma \ref{lem_poinc}]
As before, for the sake of simplicity, throughout the proof the constant $C$ is meant to be any constant that may be determined in terms of $n$ alone. The proof is included for completeness (with a slightly different notation can be found also in \cite{Fu1}).

We recall (see for instance \cite{Mu}) that
$v(\xi)\in L^2(\Ss^{n-1}) $ admits a Fourier expansion, in the sense that there exists 
a family of spherical harmonics $\{Y_k(\xi)\}_{k\in\N}$ which satisfy for all $k\ge 0$ 
$$-\Delta_\xi Y_k = k(k+n-2)Y_k  \quad\mbox{and}\quad \|Y_k\|_{L^2}=1,$$
such that
$$v(\xi)=\sum_{k=0}^{+\infty}a_kY_k(\xi)\qquad \xi \in {\Ss}^{n-1}$$
and the coefficients $a_k$ are the projections of $v$ onto the subspaces spannedd by $Y_k$:
$$a_k=\int_{\mathcal{S}^{n-1}}v(\xi) Y_k(\xi)\, d\sigma_\xi,$$
so that
$$\|v\|^2_{L^2}=\sum_{k=0}^{+\infty}a^2_k.$$
Notice that $Y_0=(n\omega_n)^{-1/2}$, therefore \eqref{eq_volume} implies 
$$|a_0|=(n\omega_n)^{-1/2}\left|\int_{{\Ss}^{n-1}}v(\xi)\, d\sigma_\xi\right|\le(n\omega_n)^{-1/2}\left|\int_{{\Ss}^{n-1}}v^2(\xi)\, d\sigma_\xi\right|\left(\frac{n-1}{2}+C\ep \right).$$
Up to renaming the constant $C$, we get 
$$|a_0|\le C\varepsilon \|v\|_{L^2}.$$

Now, multiplying the identity

$$(1+v)^{n+1}=\sum_{k=0}^{n+1}\binom nkv^k,$$
 
by $Y_1$ and integrating we get  

\begin{equation}\label{eq_y1}
\left |\int_{\Ss^{n-1}}Y_1(\xi)(1+v(\xi))^{n+1}\, d\sigma_\xi\right|\ge\left |\int_{\Ss^{n-1}}Y_1(\xi)v(\xi)(n+1)\, d\sigma_\xi\right| - C\varepsilon \|v\|_{L^2}.
\end{equation}

Here we used H\"older's inequality and the fact that the integral of $Y_1$ over $\Ss^{n-1}$ is zero. 

Since in our notation the function $Y_1(\xi)=\xi\cdot\tau$ for some suitable vector $\tau$, then  the \emph{Barycenter constraint} 3.) implies that the lefthand side of \eqref{eq_y1} vanishes and 
$$|a_1|=\left |\int_{\Ss^{n-1}}Y_1(\xi)v(\xi)\, d\sigma_\xi\right| \le C\varepsilon \|v\|_{L^2}.$$

Finally we can conclude the proof observing on one hand that

$$\|v\|_{L^2}^2=\sum_{k=0}^{\infty}a_k^2=a_0^2+a_1^2+\sum_{k=2}^{\infty}a_k^2\le  C\ep \|v\|_{L^2}^2 +\sum_{k=2}^{\infty}a_k^2$$

and on the other hand

$$\|Dv\|_{L^2}^2=\sum_{k=1}^{\infty}k(k+n-2)a_k^2\ge \sum_{k=2}^{\infty}k(k+n-2)a_k^2\ge 2n \sum_{k=2}^{\infty}a_k^2.$$

\end{proof}
 


\begin{proposition}\label{main_prop}
{For any given $n\in \N$ and $\omega>0$, there exist two positive constants $K>0$ and $0<\varepsilon_0<1$ depending on $n$ and $\omega$ only, such that for all $0<\varepsilon<\ep_0$ and $\Omega\in(n,\omega,\ep)$--NS, then
\begin{eqnarray*}
\dfrac{V(\Omega^\sharp)P(\Omega)-P(\Omega^\sharp)V(\Omega)}{n\omega_n}&=&\\
f_\raggio(1)\is h_\raggio(1+v(\xi))\,(1+v)^{n-1}\sqrt{1+\frac{|Dv(\xi)|^2}{(1+v(\xi))^2}}\,d\sigma_\xi&-& \\
h_\raggio(1)\is f_\raggio(1+v(\xi))\,(1+v(\xi))^{n-1}\,d\sigma_\xi&\ge& K\is v^2\,d\sigma_\xi
\end{eqnarray*}
whenever $v\in \mathcal{N}(n,\varepsilon)$ provides the polar representation of $\partial\Omega$ and $h_\raggio$ and $f_\raggio$ are the functions defined in \eqref{eq_h} and \eqref{eq_f}.}
\end{proposition}

\begin{proof}[Proof of Proposition \ref{main_prop}] 

In what follows $K_1$ and $K_2$ are meant to be constants that may be determined in terms of $n$ and $\omega$ alone.

By using Lemmata \ref{lem_expans}, \ref{lem_taylor} and \ref{lem_poinc} we have
\begin{align*}
&f_\raggio(1)\is h_\raggio(1+v)\,(1+v)^{n-1}\sqrt{1+\frac{|Dv|^2}{(1+v)^2}}\,d\sigma- h_\raggio(1)\is f_\raggio(1+v)\,(1+v)^{n-1}\,d\sigma\\
\ge& \is v\left[ f_\raggio(1)h_\raggio'(1)-f_\raggio'(1)h_\raggio(1)\right] \,d\sigma\\
&+ \is \frac{v^2}{2} \left[ f_\raggio(1)h_\raggio''(1)-f_\raggio''(1)h_\raggio(1) + 2(n-1) \left( f_\raggio(1)h_\raggio'(1)-f_\raggio'(1)h_\raggio(1)\right) \right]  \,d\sigma\\
&+ \is h_\raggio(1)f_\raggio(1)\frac{|Dv|^2}{2} \,d\sigma - \varepsilon K_1\|Dv\|_{L^2}^2\\
\ge& \frac{1}{4n}\Big[ (n-1) \left( f_\raggio(1)h_\raggio'(1)-f_\raggio'(1)h_\raggio(1) \right)  + \left( f_\raggio(1)h_\raggio''(1)-f_\raggio''(1)h_\raggio(1) \right)\\
 &+2nh_\raggio(1)f_\raggio(1)-\varepsilon K_2\Big] \|Dv\|_{L^2}^2
\end{align*}
provided $\ep\le\ep_0$ and $\ep_0$ is small enough.

Since $K_2$ does not depend on $\varepsilon$ then $\ep K_2$ can be chosen arbitrarily small provided $\ep_0$ is small enough and the proof is complete if we just prove that 
\begin{align}\label{eq_cfinale}
&(n-1) \left( f_\raggio(1)h_\raggio'(1)-f_\raggio'(1)h_\raggio(1) \right)+ \left( f_\raggio(1)h_\raggio''(1)-f_\raggio''(1)h_\raggio(1) \right)\\\notag
 & +2nh_\raggio(1)f_\raggio(1)>0
\end{align}
for all $\raggio>0$ (i.e. for all $\omega>0$).

Let us define $\nu=\frac n2-1$ and for convenience let us rewrite inequality \eqref{eq_cfinale} as:  
\begin{equation}\label{eq_t}
\left[t^{(2\nu+1)}\left(f_\raggio(t)\frac{d}{dt}h_\raggio(t)-h_\raggio(t)\frac{d}{dt}f_\raggio(t)\right)\right]'+4(\nu+1)t^{(2\nu-1)}h_\raggio(t)f_\raggio(t)>0
\end{equation}
in $t=1$ and for all $\raggio>0$.

We multiply the left hand side of \eqref{eq_t} by $\raggio^{2\nu-1}$, and after the change of variables $s=\raggio t$
 the inequality becomes

\[
\begin{split}
&\frac{d}{ds}\left[s^{(2\nu+1)}\left(f_\raggio\left(\frac s\raggio\right)\frac{d}{ds}h_\raggio\left(\frac s\raggio\right)-h_\raggio\left(\frac s\raggio\right)\frac{d}{ds}f_\raggio\left(\frac s\raggio\right)\right)\right]\\
&+4(\nu+1)s^{(2\nu-1)}h_\raggio\left(\frac s\raggio\right)f_\raggio\left(\frac s\raggio\right)>0
\end{split}
\]

to be proven in $s=\raggio$ and for all $\raggio>0$.
 
According to \eqref{eq_h} and \eqref{eq_f} then

$$h_\raggio\left(\frac s\raggio\right)=s^{-2\nu}\I^2_\nu(s)$$

and

$$f_\raggio\left(\frac s\raggio\right)=s^{-2\nu}\I_\nu(s)\I_{\nu+1}(s).$$




In the following we will repeatedly use \eqref{besseld} and \eqref{besseld2}, according to which 

\[
\frac {d}{ds}\left(s^{-\nu}\I_\nu\right)=s^{-\nu}\I_{\nu+1}
\]
and
\[
\frac {d}{ds}\left(s^{\nu+1}\I_{\nu+1}\right)=s^{\nu+1}\I_{\nu}.
\]
Then we can easily determine that
\begin{eqnarray*}
&&s^{(2\nu+1)} \left(f_\raggio\left(\frac s\raggio\right)\frac{d}{ds}h_\raggio\left(\frac s\raggio\right)-h_\raggio\left(\frac s\raggio\right)\frac{d}{ds}f_\raggio\left(\frac s\raggio\right)\right) \\\notag\\
&&\quad=\displaystyle{{s\,\I_\nu^2\left(s\right)\,\I_{\nu+1}^2\left(s\right)+\left(2\,
 \nu+1\right)\,\I_\nu^3\left(s\right)\,{\I_{\nu+1}}\left(s\right)-s\,
 \I_\nu^4\left(s\right)}\over{s^{2\,\nu}}}.
 \end{eqnarray*}
 
Then, differentiating again by $s$ and summing $4(\nu+1)s^{(2\nu-1)}h_\raggio\left(\frac s\raggio\right)f_\raggio\left(\frac s\raggio\right)$ we have
\begin{eqnarray*}
&&\frac{d}{ds}\left[s^{(2\nu+1)}\left(f_\raggio\frac{d}{ds}h_\raggio-h_\raggio\frac{d}{ds}f_\raggio\right)\right]+4(\nu+1)s^{(2\nu-1)}h_\raggio f_\raggio\\ \notag\\
&&\qquad=s^{-(2\nu+1)}\Big[ 2\,s^2\,{\I_\nu}\left(s\right)\,\I_{\nu+1}^3\left(s\right)+2\left(
 2\,\nu+1\right)\,s\,\I_\nu^2\left(s\right)\,\I_{\nu+1}^2\left(s
 \right)\notag\\
&&\qquad\quad+\left(\left(2\nu+3\right)-2\,s^2\right)\,\I_\nu^3\left(s
 \right)\,{\I_{\nu+1}}\left(s\right)\Big].\notag
\end{eqnarray*}

Finally \eqref{eq_cfinale} holds true for all $\raggio>0$ and $n\ge2$ provided the function
 $$H_n(s)=2\,s^2\,\I_{\nu+1}^2\left(s\right)+2\left(2\,\nu+1\right)\,s\,{\I_\nu}
 \left(s\right)\,{\I_{\nu+1}}\left(s\right)+\left(-2\,s^2+2\,\nu+3\right)\,
 {\I_\nu}^2\left(s\right)> 0$$
for all $s>0$ and $n\ge2$.
 
We can prove such an inequality for instance by differentiating with respect to $s$
$$ \frac{d}{ds}{H_n}(s)=\frac 1s\left({H_n+4s{\I_\nu}{\I_{\nu+1}}+(4\nu^2+4\nu-1){\I_\nu}^2}\right).$$
Observing that $(4\nu^2+4\nu-1)\ge 0$ if $n\ge 3$ and that $\I_\nu(s)$ and $\I_{\nu+1}(s)$ are positive for all $s>0$, 
we deduce $s {H_n}'(s)\ge {H_n}(s)$ provided $n\ge 3$. This immediately implies $H_n(s)>0$ for all $s>0$ and $n\ge 3$.

For $n=2$ we set $\displaystyle G(s)=s\frac{d}{ds}{H_2}(s),$ and differentiating once again
$$ \frac{d}{ds}G(s)=6s{\I_{\nu+1}}^2+2s{\I_\nu}^2> 0\qquad \forall s>0,$$
we get $H_2>0$ for all $s>0$.
\end{proof}


{We can now conclude the proof of Theorem \ref{teonearly}. For given $n\in \N$ and $\omega>0$ let $0<\ep_0<1$ and $K$ be the constant given by Proposition \ref{main_prop}. Let $0<\ep<\ep_0$,  and let $\Omega$ be any set in $(n,\omega,\ep)$--NS. As usual let $v\in \mathcal{N}(n,\ep)$ be the function providing the polar representation of $\partial\Omega$. We then use Lemma \ref{lem_stima}, Proposition \ref{main_prop}, the monotonicity of $h$ and $f$, and the boundedness of the function $v$ to get
\begin{eqnarray*}
\lambda(\Omega^{\sharp})&=& \frac{V(\Omega^\sharp)}{P(\Omega^\sharp)} \ge \frac{V(\Omega)}{P(\Omega)}+\frac{n\omega_nK\is v^2\,d\sigma_\xi}{P(\Omega^\sharp) P(\Omega)}\\
&=& \frac{V(\Omega)}{P(\Omega)}\left(1+\frac{n\omega_nK\is v^2\,d\sigma_\xi}{P(\Omega^\sharp) V(\Omega)}\right)\\
&=& \frac{V(\Omega)}{P(\Omega)}\left(1+\frac{K\is v^2\,d\sigma_\xi}{h_\raggio(1)\is f_\raggio(1+v(\xi))\,(1+v(\xi))^{n-1}\,d\sigma_\xi}\right)\\
&\ge&  \frac{V(\Omega)}{P(\Omega)}\left(1+\frac{K\is v^2\,d\sigma_\xi}{n \omega_n 2^{n-1}h_\raggio(1) f_\raggio(2) }\right),\\
&\ge& \lambda(\Omega)\left(1+\frac{K\is v^2\,d\sigma_\xi}{n \omega_n 2^{n-1}h_\raggio(1) f_\raggio(2) }\right),\\
\end{eqnarray*}}
which is exactly the inequality in the statement on Theorem \ref{teonearly} provided $K$ is renamed.

\section{Proof of Theorem \ref{mainteo}}\label{s_main}

First we use the same arguments employed in the proof of Lemma \ref{lem_stima} and in particular the estimate given in \eqref{eq_stima2} to observe that if $\Omega$ is any bounded Lipschitz subset of $\R^n$ then
$$\lambda(\Omega)\le \frac{V(\Omega)}{P(\Omega)},$$
where $V(\Omega)$ and $P(\Omega)$ are the weighted volume and perimeter introduced in Section \ref{introduction}. Namely, using \eqref{besseld}, we have
\begin{equation}\label{eq_wv}
\displaystyle V(\Omega)=  \int_\Omega |x|^{2-n}\left(\I^2_{\frac{n}{2}-1}(|x|)+\I^2_{\frac{n}{2}}(|x|)\right)dx,
\end{equation}
and
\begin{equation}\label{eq_wp}
\displaystyle P(\Omega)=\int_{\partial \Omega} |x|^{2-n}\,\I^2_{\frac{n}{2}-1}(|x|)\,d\h.
\end{equation}
We observe that the notion of weighted volume and weighted perimeter in \eqref{eq_wv} and \eqref{eq_wp} can be extended to the whole class of sets of finite perimeter, provided topological boundary in \eqref{eq_wp} is replaced by the reduced boundary $\partial^*\Omega$ (see \cite{Ma}).

%
%
%
%
%

For any measurable set $E$ we shall denote by $X(E)$ its barycenter
and, for any given $\omega>0$, we put 
$$\gamma_\omega=\frac{P(B_\raggio)}{V(B_\raggio)},$$
the reciprocal of the eigenvalue $\lambda(B_\raggio)$, remembering that $\raggio=(\omega/\omega_n)^{1/n}$.
Theorem \ref{mainteo} is then a consequence of the following result.

\begin{theorem}\label{teo_weighted}
For any given $n\in\N$ and $\omega>0$ there exists a positive constant $\delta$ such that, in the class of finite perimeter sets, the ball $B_\raggio$ is the unique minimizer of  

\begin{equation}\label{constrained}
\min_\Omega\{J_0(\Omega): \m(\Omega)=\omega,\>B_{\raggio-\delta}\subset\Omega\subset B_{\raggio+\delta},\>X(\Omega)=0\},
\end{equation}
where
\begin{equation}
J_0(\Omega)=P(\Omega)-\gamma_\omega V(\Omega).
\end{equation}
\end{theorem}

According to what we observed in the introduction concerning the choice of the reference system, without the barycenter condition $X(\Omega)=0$ the result is false. In principle this mandates the additional hypothesis in the statement of Theorem \ref{mainteo} 
$$X(\Omega)=0.$$ 
However, by trivial arguments one can easily prove that, if Theorem \ref{mainteo} is true under such an additional constraint then, possibly after taking a smaller positive $\delta$, it holds true even without it.

%
%


The proof of Theorem \ref{teo_weighted} is done in two steps : Proposition \ref{minima} and Proposition \ref{p_unconstrained}. We start by showing that the minimizers for \eqref{constrained} are actually the minimizers for an unconstrained minimum problem involving a suitable penalized functional, namely, we have:

\begin{proposition}\label{minima}
For every $n\in \N$ and $\omega>0$ there exist some positive constants, $\delta_0$, $\Lambda_1$, $\Lambda_2$, $\Lambda_3$, 
such that, for every $0<\delta<\delta_0$, the unconstrained functional
\begin{equation}\label{J}
J(\Omega)=J_0(\Omega)+\Lambda_1\,|X(\Omega)|+\Lambda_2\,|\m(\Omega)-\omega|+\Lambda_3(V(\Omega\backslash B_{\raggio+\delta})+V(B_{\raggio-\delta}\backslash \Omega)),
\end{equation}
admits minimizers in the class of the sets with finite measure and perimeter, and $\Omega$ is a minimizer for \eqref{constrained} if and only if is a minimizer for \eqref{J}.
\end{proposition}

\begin{proof}
We proceed by steps.

\vskip.4cm
\noindent {\bf Claim 1.} 
There exist two positive constants $\Lambda_1$ and $\delta_0$, which depends on $\omega$ and $n$ only, such that, for all $0<\delta<\delta_0
$,
there exists minimizers for the problem
\begin{equation}\label{J1}
\min_\Omega\{J_0(\Omega)+\Lambda_1\,|X(\Omega)|:\m(\Omega)=\omega,\>B_{\raggio-\delta}\subset\Omega\subset B_{\raggio+\delta}\}
\end{equation}
in the class of finite perimeter sets, and every minimizer is also a minimizer for \eqref{constrained}.
\vskip.3cm

Since the existence of minimizers for \eqref{J1} is trivial, it is enough to show that there exist positive constants $\delta_0$ and $\Lambda_1$ depending on $n$ and $\omega$ only, such that whenever for some $a>0$ and $\delta<\delta_0$ a set $\Omega$ satisfies
\begin{enumerate}[(a)]
\item $\m(\Omega)=\omega$, 
\item $B_{\raggio-\delta}\subset\Omega\subset B_{\raggio+\delta}$
\item $X(\Omega)=a \,e_1 $, with $e_1=(1,0,\dots,0)$
\end{enumerate}
then it is possible to find a set $\tilde\Omega$ satisfying the constraint (a) and (b) and moreover $J_0(\tilde\Omega)+\Lambda_1\,|X(\tilde\Omega)|< J_0(\Omega)+\Lambda_1\,|X(\Omega)|$.

Firstly we show that there exists $\delta_0$ small enough, such that for any $\Omega$ satisfying (a),(b), and (c) for some $\delta<\delta_0$, then certainly 
$$\m\{x\in(\Omega\setminus B_{\raggio-\delta});\;x\cdot e_1>6\delta_0\}>0,$$
$$\m\{x\in(B_{\raggio+\delta}\setminus \Omega);\;x\cdot e_1<-6\delta_0\}>0.$$
Indeed we have 
\begin{itemize}
\item $\m(B_{\raggio+\delta}\setminus B_{\raggio-\delta})=2 n \omega_n \raggio^{n-1}\delta+O(\delta^2)$
\item $\m(\Omega\setminus B_{\raggio-\delta})=n \omega_n \raggio^{n-1}\delta+O(\delta^2)$
\item $\m\{x\in(B_{\raggio+\delta}\setminus B_{\raggio-\delta});\; |x\cdot e_1|<6\delta_0\}=2 (n-1) \omega_{n-1} \raggio^{n-2}(\delta+O(\delta^2))(6\delta_0+O(\delta_0^2))$
\end{itemize}

Therefore for $\delta_0$ small enough
$$\frac{9}{10}n \omega_n \raggio^{n-1}\delta\le \m(\Omega\setminus B_{\raggio-\delta})\le \frac{11}{10}n \omega_n \raggio^{n-1}\delta$$
$${ \frac{9}{10}n \omega_n \raggio^{n-1}\delta\le}\frac12\m(B_{\raggio+\delta}\setminus B_{\raggio-\delta})\le \frac{11}{10}n \omega_n \raggio^{n-1}\delta$$
$$\m\{x\in(B_{\raggio+\delta}\setminus B_{\raggio-\delta});\; |x\cdot e_1|<6\delta_0\}\le \frac{3}{10}n \omega_n \raggio^{n-1}\delta.$$

As a consequence assuming $$\m\{x\in(\Omega\setminus B_{\raggio-\delta});\;x\cdot e_1>6\delta_0\}=0,$$
then $$\m\{x\in(\Omega\setminus B_{\raggio-\delta});\;x\cdot e_1<-6\delta_0\}\ge \frac{6}{10}n \omega_n \raggio^{n-1}\delta\ge \frac{1}{2}\m(\Omega\setminus B_{\raggio-\delta}),$$
in contradiction with (c). 

Similarly if we assume $$\m\{x\in(B_{\raggio+\delta}\setminus \Omega);\;x\cdot e_1<-6\delta_0\}=0$$
then 
$$\m\{x\in(\Omega\setminus B_{\raggio-\delta});\;x\cdot e_1<-6\delta_0\}\ge\frac{6}{10}n \omega_n \raggio^{n-1}\delta\ge \frac{1}{2}\m(\Omega\setminus B_{\raggio-\delta})$$
again contradicting (c).

From now on we consider $\delta_0$ fixed in term of $n$, $\omega$, and in what follows we name $C_i$, $i=1,2,...$, generic constants which depends on $n$, $\omega$ alone. 

Let us denote by $E_0$ the union of all balls of radii $3\delta_0$ not intersecting $B_{\raggio-\delta}\cap \{x\in \R^n;\,x\cdot e_1 \le 3\delta_0\}$, and 
by $F_0$ the union of all balls of radii $3\delta_0$ included in $B_{\raggio+\delta}\cap \{x\in\R^n;\, x\cdot e_1 < 0\}$.

Then we define $E_t=\{x\in E_0; \dist(x,\partial E_0)>t, \,t>0\}$ and  $F_t=\{x\in F_0; \dist(x,\partial F_0)>t, \,t>0\}$, the level sets of the distance functions from the boundary of $E_0$ and $F_0$ respectively. 
Clearly, regardless the choice of $\Omega$ satisfying (a), (b), (c) we have 
$$\m(E_0\cap \Omega)>0$$ 
$$\m(F_0\setminus \Omega)>0$$
$$\m(E_{2\delta}\cap \Omega)=0$$ 
$$\m(F_{2\delta}\setminus \Omega)=0$$

We will prove that for any given $\Omega$ satisfying (a), (b), and (c) there exists a positive constant $\sigma_0$ such that for all $0<\sigma<\sigma_0$ we can find $\bar t , \tilde t \in[0,2\delta]$ depending on $\sigma$
such that the following properties hold:
\begin{enumerate}[(i)]
\item $\m(\Omega\cap E_{\bar t})=\m(F_{\tilde t}\setminus \Omega)=\sigma$
\item $\m(E_{\bar t}\cap B_{\raggio-\delta})=\m(F_{\tilde t}\setminus B_{\raggio+\delta})=0$
\item $P(\Omega\setminus E_{\bar t})< P(\Omega) + C_1\sigma$
\item $P(\Omega\cup F_{\tilde t})< P(\Omega) + C_1\sigma$
\end{enumerate}

Properties (i)--(ii) follow at once by the definition of $E_t$ and $F_t$. 

Concerning (iii) we observe that, for $t\in(0,2\delta)$,  $E_t$ and $F_t$ are two families of $C^{1,1}$ sets with mean curvature bounded in terms of $\delta_0$ and $\omega$ and $n$ alone. Moreover, if for $x\in E_0$ such that $\dist (x,\partial E_0)\le 2\delta_0$ we set $T_{E}(x)=\nabla \dist(x,\partial E_0)$, then for all $x\in\partial E_t $ the vector $T_E(x)$ is the inner unit normal to $\partial E_t$ at $x$. Since for $\dist (x,\partial E_0)\le 2\delta_0$ we have
$$\left|{\rm div}(|x|^{2-n}\,\I^2_{\frac{n}{2}-1}(|x|)T_E(x))\right|\le\left| \nabla(|x|^{2-n}\,\I^2_{\frac{n}{2}-1}(|x|))\right|+ |x|^{2-n}\,\I^2_{\frac{n}{2}-1}(|x|)\left|{\rm div}T_E\right|\le C_1$$ 
we have

$$|P(\Omega\setminus E_{\bar t})-P(\Omega)| \le \left|\int_{\Omega\cap E_{\bar t}} {\rm div}\left(|x|^{2-n}\,\I^2_{\frac{n}{2}-1}(|x|)T_E(x)\right) dx \right|\le C_1\m(\Omega\cap E_{\bar t}).$$

We can argue in the same way to deduce (iv) by using $F_0$ in place of $E_0$ and by defining the vector field $T_F(x)=\nabla \dist(x,\partial F_0)$ for $\dist (x,\partial F_0)\le 2\delta_0$. We have 
$$|P(\Omega\cup F_{\tilde t})-P(\Omega)| \le \left|\int_{ F_{\tilde t} \setminus \Omega} {\rm div}\left(|x|^{2-n}\,\I^2_{\frac{n}{2}-1}(|x|)T_F(x)\right) dx \right|\le C_1\m(F_{\tilde t}\setminus\Omega).$$

Finally we observe that $V(\Omega\cup F_{\tilde t})> V(\Omega)$, and using \eqref{bessel_1}--\eqref{bessel_2}, then we also have $V(\Omega\setminus E_{\bar t})\ge V(\Omega) - C_2\sigma$. \\
Once we have constructed the sets $E_{\bar t}$ and $F_{\tilde t}$ the proof of the claim follows at once from (i)--(iv). Indeed let 
$$\tilde\Omega=(F_{\tilde t}\cup\Omega)\setminus E_{\bar t},$$ we have by construction
$$|X(\tilde\Omega)|\le |X(\Omega)|-C_3\sigma+O(\sigma^2)\le|X(\Omega)|-C_4\sigma$$
provided $\sigma$ is smaller than a positive quantity $\sigma_0$ which may depend on $\Omega$. Eventually we have
\begin{eqnarray*}
J_0(\tilde\Omega)+\Lambda_1\,|X(\tilde\Omega)|&=& P(\tilde\Omega)-\gamma_\omega V(\tilde\Omega)+\Lambda_1\,|X(\tilde\Omega)|\\
& \le& P(\Omega)-\gamma_\omega V(\Omega)+\Lambda_1\,|X(\Omega)|+(2 C_1+C_2 \gamma_\omega-C_4\Lambda_1)\sigma\\
& <& J_0(\Omega)+\Lambda_1\,|X(\Omega)|\\
\end{eqnarray*}
provided $\Lambda_1$ is large enough. Observe that $\Lambda_1$ can be explicitly computed in terms of $n$, $\omega$ and $\delta_0$ alone.

\vskip.4cm
\noindent {\bf Claim 2.} Let $\delta_0$ and $\Lambda_1$ be the constant given in Claim 1, then there exists a positive constant $\Lambda_2$, which depends on $\omega$ and $n$ only, such that, for every $0<\delta<\delta_0$,
 the problem
\begin{equation}\label{J2}
\min_\Omega\{J_0(\Omega)+\Lambda_1|X(\Omega)|+\Lambda_2\,|\,|\m(\Omega)-\omega|:B_{\raggio-\delta}\subset\Omega\subset B_{\raggio+\delta}\}
\end{equation}
admits minimizers in the class of finite perimeter sets, and every minimizer is also a minimizer for \eqref{J1}.
\vskip.3cm
As in Claim 1 existence of minimizers is trivial and arguing as before, let us assume that for some $\delta\in(0,\delta_0)$, a set $\Omega$ satisfies $$B_{\raggio-\delta}\subset\Omega\subset B_{\raggio+\delta}$$ and $$\m(\Omega)>\omega.$$
In this case there exists $\tilde \raggio\in(\raggio-\delta,\raggio+\delta)$ such that $\m(\Omega\cap B_{\tilde \raggio})=\omega$. 
If we set $$\tilde \Omega=\Omega\cap B_{\tilde \raggio}$$ then
$$|P(\tilde \Omega)-P(\Omega)|\le \left|\int_{\Omega\setminus \tilde\Omega} {\rm div}\left(|x|^{2-n}\,\I^2_{\frac{n}{2}-1}(|x|)\frac{x}{|x|}\right) dx \right|\le C_5\m(\Omega\setminus \tilde\Omega).$$
As in the previous claim by $C_i$ we denote positive constants depending just on $\delta_0$, $n$, and $\omega$.\\
Using 
\eqref{bessel_1} we deduce that 
$$V(\tilde\Omega)\ge V(\Omega) - C_6\m(\Omega\setminus \tilde\Omega).$$
Finally we observe that $$|X(\tilde\Omega)|\le|X(\Omega)|+(\raggio+\delta_0)\m(\Omega\setminus \tilde\Omega),$$
and therefore

\begin{eqnarray*}
J_0(\tilde\Omega)+\Lambda_1\,|X(\tilde\Omega)|+ \Lambda_2\,|\m(\tilde\Omega)-\omega|&=& P(\tilde\Omega)-\gamma_\omega V(\tilde\Omega)+\Lambda_1\,|X(\tilde\Omega)|\\
& \le& P(\Omega)-\gamma_\omega V(\Omega)+\Lambda_1\,|X(\Omega)|+\Lambda_2\,|\m(\Omega)-\omega|\\
&&+(C_5+C_6 \gamma_\omega+(\raggio+\delta_0)\Lambda_1-\Lambda_2) \m(\Omega\setminus \tilde\Omega)\\
& <& J_0(\Omega)+\Lambda_1\,|X(\Omega)|+\Lambda_2\,|\m(\Omega)-\omega|\\
\end{eqnarray*}
provided $\Lambda_2$ is large enough. Observe that $\Lambda_2$ can be explicitly computed in terms of $n$, $\omega$, $\delta_0$ and $\Lambda_1$ alone. 

If now we assume that for some $\delta\in(0,\delta_0)$ a set $\Omega$ satisfies $$B_{\raggio-\delta}\subset\Omega\subset B_{\raggio+\delta}$$ and $$\m(\Omega)<\omega,$$ the previous arguments work as well provided $\tilde\Omega=\Omega\cup B_{\tilde \raggio}$ for some
$\tilde \raggio\in(\raggio-\delta,\raggio+\delta)$ such that $\m(\tilde\Omega)=\omega$.

\vskip .4cm
\noindent {\bf Claim 3.} Let $\delta_0$, $\Lambda_1$ and $\Lambda_2$ be the constants given in Claim 1 and Claim 2, then there exist a positive constant $\Lambda_3$, which depends on $\omega$ and $n$ only, such that, for every $0<\delta<\delta_0$,
the problem
\begin{equation}\label{J3}
\min_\Omega\{J_0(\Omega)+\Lambda_1|X(\Omega)|+\Lambda_2\,|\,\m(\Omega)-\omega|+\Lambda_3(V(\Omega\backslash B_{\raggio+\delta})+V(B_{\raggio-\delta}\backslash \Omega))\}
\end{equation}
admits minimizers in the class of the sets with finite measure and perimeter, and every minimizer is a minimizer for \eqref{J2}.
\vskip .3cm

In the same spirit of the previous claims let us consider a set  $\Omega$ which for some $\delta\in(0,\delta_0)$ satisfies
$$V(\Omega\backslash B_{\raggio+\delta})+V(B_{\raggio-\delta}\backslash \Omega)>0.$$

We set $$\tilde\Omega=(\Omega\cap B_{\raggio+\delta})\cup B_{\raggio-\delta}.$$

If $s(t)$ is any smooth function such that 
\begin{itemize}
\item $0\le s(t)\le 1$ for all $t\ge 0$
\item $s(t)=1$ for $|t-1|\le\frac14$
\item $s(t)=0$ for $|t-1|\ge\frac12$
\item $\|s\|_{C^1}$ is bounded by a constant $K$
\end{itemize}
then
\begin{eqnarray*}
|P(\tilde \Omega)-P(\Omega)|&\le& \left|\int_{\Omega\setminus B_{\raggio+\delta}} {\rm div}\left(s\left(\frac{|x|}{\raggio}\right)|x|^{2-n}\,\I^2_{\frac{n}{2}-1}(|x|)\frac{x}{|x|}\right) dx \right|\\
&&+ \left|\int_{B_{\raggio-\delta}\setminus\Omega} {\rm div}\left(s\left(\frac{|x|}{\raggio}\right)|x|^{2-n}\,\I^2_{\frac{n}{2}-1}(|x|)\frac{x}{|x|}\right) dx \right|\\
&\le& C_7\left(\m({B_{\raggio-\delta}\setminus\Omega})+\m(\Omega\setminus B_{\raggio+\delta})\right).
\end{eqnarray*}
Here we have used that $s(|x|/\raggio)=1$ on $B_{\raggio+\delta}$ and $B_{\raggio-\delta}$ which is true provided $\delta_0 < \raggio/4$. An assumption that we can always effort without loss of generality.

Observe also that 
$$\m({B_{\raggio-\delta}\setminus\Omega})+\m(\Omega\setminus B_{\raggio+\delta})\le C_8 (V({B_{\raggio-\delta}\setminus\Omega})+V(\Omega\setminus B_{\raggio+\delta})),$$
and using  \eqref{bessel_1} and \eqref{bessel_2}
\begin{eqnarray*}
|X(\tilde\Omega)|-|X(\Omega)| &\le& \raggio\,\m({B_{\raggio-\delta}\setminus\Omega}) + \int_{\Omega\setminus B_{\raggio+\delta}} |x|\,dx \\
&\le& \raggio \m(\Omega\setminus B_{\raggio+\delta}) + C_9 \left|\int_{\Omega\setminus B_{\raggio+\delta}} |x|^{2-n}\,\I^2_{\frac{n}{2}-1}(|x|)  \, dx \right|\\
&\le& C_{10} (V({B_{\raggio-\delta}\setminus\Omega})+V(\Omega\setminus B_{\raggio+\delta})).
\end{eqnarray*}

Now we have
\begin{eqnarray*}
&&J_0(\tilde\Omega)+\Lambda_1\,|X(\tilde\Omega)|+ \Lambda_2\,|\m(\tilde\Omega)-\omega|+\Lambda_3(V(\tilde\Omega\backslash B_{\raggio+\delta})+V(B_{\raggio-\delta}\backslash \tilde\Omega))\\
&&\quad= P(\tilde\Omega)-\gamma_\omega V(\tilde\Omega)+\Lambda_1\,|X(\tilde\Omega)|+ \Lambda_2\,|\m(\tilde\Omega)-\omega|\\
&&\quad\le P(\Omega)-\gamma_\omega V(\Omega)+\Lambda_1\,|X(\Omega)|+\Lambda_2\,|\m(\Omega)-\omega|+\Lambda_3(V(\Omega\backslash B_{\raggio+\delta})+V(B_{\raggio-\delta}\backslash \Omega))\\
&&\qquad\quad+(C_7C_8+C_8\gamma_\omega+ \Lambda_1 C_{10}+C_8\Lambda_2-\Lambda_3) (V(\Omega\backslash B_{\raggio+\delta})+V(B_{\raggio-\delta}\backslash \Omega))\\
&&\quad< J_0(\Omega)+\Lambda_1\,|X(\Omega)|+\Lambda_2\,|\m(\Omega)-\omega|+\Lambda_3(V(\Omega\backslash B_{\raggio+\delta})+V(B_{\raggio-\delta}\backslash \Omega)),
\end{eqnarray*}
provided $\Lambda_3$ is large enough. Once again $\Lambda_3$ can be expressed in terms of $n$, $\omega$, $\delta_0$, $\Lambda_1$ and $\Lambda_2$.

Therefore, if a minimizer of \eqref{J3} exists then it is necessarily bounded. Since a minimizing sequence of equibounded sets for \eqref{J3} certainly exists, by compactness and semicontinuity existence of minimizers trivially follows. 

\vskip .4 cm
Summing up Claim 1, Claim 2 and Claim 3, 
it is therefore possible to choose positive constants $\delta_0$, $\Lambda_1$, $\Lambda_2$ and $\Lambda_3$,
in such a way that \eqref{J} admits minimizers whenever $0<\delta<\delta_0$ and every minimizer is also a minimizer of \eqref{constrained}. 

Finally for the same choice of the constants $\delta_0$, $\Lambda_1$, $\Lambda_2$ and $\Lambda_3$, all minimizers of \eqref{constrained} for some $0\le\delta<\delta_0$ are also a minimizers of \eqref{J}.\end{proof}

\begin{proposition}\label{p_unconstrained}

Let $\omega>0$ and let $\delta_0$, $\Lambda_1$, $\Lambda_2$, $\Lambda_3$ be the positive constants given in Proposition \ref{minima}. There exists $\delta$ with $0<\delta<\delta_0$ such that any minimizer in the class of the sets with finite measure and perimeter of the functional $J$ defined in \eqref{J} is a ball.
\end{proposition}
\begin{proof}
Let  $\Omega$ be a minimizer of the functional $J$ defined in \eqref{J} for a fixed value of $\delta$ with $0<\delta<\delta_0$. Our first aim is to prove that $\Omega$ is an almost-minimizer for the perimeter in the sense of \eqref{almminp}.

Since $\Omega$  is a minimizer of $J$, in view of Proposition \ref{minima}, we have $J(\Omega)=J_0(\Omega)$, the barycenter of $\Omega$ is the origin, and $B_{\raggio-\delta}\subset\Omega\subset B_{\raggio+\delta}$. Let now $0<r_0<1$,  
 we consider a set $\tilde\Omega$ such that $\Omega\dsim\tilde\Omega\subset\subset B_r(x_0)$ for some and $0<r<r_0$.
If $|x_0|>\raggio+\delta+1$, then
$$P(\Omega)\le P(\tilde\Omega),$$
while if $|x_0|\le \raggio+\delta+1$, then
$$J(\Omega)\le J(\tilde\Omega)\le P(\tilde\Omega)-\gamma_\omega V(\Omega)+K_1 \m(B_r(x_0)),$$
for some $K_1$ which depends only on $n$ and $\omega$.
Therefore we can say that $\Omega$ is an almost-minimizer for the weighted perimeter $P(\cdot)$, in the sense that there exist two positive constants $K_1$ and $r_0$ , which depend only on $n$ and $\omega$, such that
\begin{equation}\label{almmin}
P(\Omega)\le P(\tilde\Omega)+K_1 r^n,
\end{equation}
whenever $\Omega\dsim\tilde\Omega\subset\subset B_r(x_0)$ and $0<r<r_0$.
Then, arguing as in \cite{FM}, observing that the weight $(|x|^{1-\frac n2}\I_{\frac{n}{2}-1}(|x|))^2$ which defines $P(\cdot)$ is locally a Lipschitz function, \eqref{almmin} implies \eqref{almminp}.
Using the results quoted in the introduction the reduced boundary $\partial^*\Omega$ is a $C^{1,1/2}$ hypersurface. Let us now consider a vanishing decreasing sequence of positive numbers $\delta_k$, $k\in\N$, with $\delta_k<\delta_0$, and let $\Omega_k$ be a sequence of minimizers for the functional $J$ with
$\delta=\delta_k$. Then $\Omega_k$ is a sequence of almost-minimizers for the Euclidean perimeter, with uniform constants $K$ and $r_0$, converging to $B_\raggio$ in $L^1$. It follows 
that there exists $\bar k\in\N$ such that, for $k>\bar k$,  
the boundary of $\Omega_k$
can be represented in polar coordinates as
\begin{equation}\label{regularity}
r_k(\xi)=\raggio(1+v_k(\xi)),
\end{equation}
where $v_k\rightarrow0$ in $C^{1,\alpha}(\Ss^{n-1})$ for all $\alpha\in(0,1/2)$.
{ Therefore, for $n\in N$ and $\omega>0$ let $\ep=\ep(n,\omega)$ be the positive constant given in Theorem \ref{teonearly}. In view of what we have proved so far, there exists $\bar\delta\in(0,\delta_0)$ such that for every $\delta\in(0,\bar\delta)$ any minimizer $\Omega$ of $J$ is a $(n,\omega,\ep)$-NS set. Employing Theorem \ref{teonearly} we deduce that $\Omega$ must be a ball.}

\end{proof}

\begin{proof}[Proof of Corollary \ref{corol}] In order to prove Corollary \ref{corol} it is enough to show that Conjecture \ref{conj2} is equivalent to Conjecture \ref{conj}, so that it follows from Theorem \ref{mainteo}. Here we only prove that Conjecture \ref{conj} implies Conjecture \ref{conj2} as the reverse follows using similar arguments.

Slightly modifying the notation given in the Introduction, for a given  open bounded Lipschitz set $\Omega$ we consider 
\begin{equation}\label{defini}
\lambda_{R,\alpha}(\Omega)=\mathop{\min_{w\in H^1(\Omega)}}_{ w\ne 0}  
\frac{\displaystyle\int_\Omega |Dw|^2\,dx+\alpha\displaystyle\int_{\partial\Omega} w^2\,d\h}{\displaystyle\int_{\Omega}w^2\, dx}, \end{equation}
with $\alpha<0$. Observing that $\lambda_{R,\alpha}(\Omega)$, as a function of $\alpha$, is Lipschitz continuous, monotone increasing from $]-\infty,0]$ onto $]-\infty,0]$, for every fixed $\alpha<0$ there exists $\bar\alpha<0$ such that:
\begin{equation}
\lambda_{R,\bar\alpha}(\Omega^\sharp)=\lambda_{R,\alpha}(\Omega).
\end{equation}
Our aim is to show that $\bar\alpha\le\alpha$.

Indeed, denoting by $u$ and $v$ two extremal functions in \eqref{defini} relative to $\Omega$ and $\Omega^\sharp$ , we have:
\begin{equation*}
\int_\Omega |Du|^2\,dx+\alpha\displaystyle\int_{\partial\Omega} u^2\,d\h=\lambda_{R,\alpha}(\Omega)\displaystyle\int_{\Omega}u^2\, dx
\end{equation*}
\begin{equation*}
\int_\Omega |Dw|^2\,dx+\alpha\displaystyle\int_{\partial\Omega} w^2\,d\h\ge\lambda_{R,\alpha}(\Omega)\displaystyle\int_{\Omega}w^2\, dx\qquad\forall w\in H^1(\Omega)
\end{equation*}
and
\begin{equation*}
\int_{\Omega^\sharp} |Dv|^2\,dx+\bar\alpha\displaystyle\int_{\partial\Omega^\sharp} v^2\,d\h=\lambda_{R,\bar\alpha}(\Omega^\sharp)\displaystyle\int_{\Omega^\sharp}v^2\, dx
\end{equation*}
\begin{equation*}
\int_{\Omega^\sharp} |Dz|^2\,dx+\alpha\displaystyle\int_{\partial\Omega^\sharp} z^2\,d\h\ge\lambda_{R,\bar\alpha}(\Omega^\sharp)\displaystyle\int_{\Omega^\sharp}z^2\, dx\qquad\forall z\in H^1(\Omega^\sharp)
\end{equation*}
By a rescaling with $\kappa=\sqrt{|\lambda_{R,\alpha}(\Omega)|}$ we have
\begin{equation}\label{e1}
\int_{\kappa\Omega} |Dw|^2\,dx+\int_{\kappa\Omega}w^2\, dx\ge\frac{|\alpha|}\kappa\int_{\partial(\kappa\Omega)} w^2\,d\h\qquad\forall w\in H^1(\kappa\Omega)
\end{equation}
and
\begin{equation}\label{e2}
\int_{\kappa\Omega^\sharp} |Dz|^2\,dx+\int_{\kappa\Omega^\sharp}z^2\, dx\ge\frac{|\bar\alpha|}\kappa\int_{\partial(\kappa\Omega^\sharp)} z^2\,d\h\qquad\forall z\in H^1(\kappa\Omega^\sharp)
\end{equation}
with equality holding in \eqref{e1} and \eqref{e2} for $w(x)=u(x/\kappa)$ and $z(x)=v(x/\kappa)$. It follows that
$$\lambda(\kappa\Omega)=\frac{|\alpha|}\kappa\qquad\text{and}\qquad\lambda(\kappa\Omega^\sharp)=\frac{|\bar\alpha|}\kappa$$
and then using Conjecture \ref{conj} we have
$\bar\alpha\le\alpha$.
\end{proof}

\centerline{\bf Acknowledgements}
The authors wish to thank F. Brock, D. Bucur and N. Fusco for their suggestions and for many fruitful discussions. 


\begin{thebibliography}{99}  

\bibitem{AS} M. Abramowitz, I.A. Stegun. {\it Handbook of mathematical functions with formulas, graphs, and mathematical tables}. National Bureau of Standards Applied Mathematics Series, 55, U.S. Government Printing Office, Washington, D.C. 1964

\bibitem{AFM} E. Acerbi, N. Fusco, M. Morini.
Minimality via second variation for a nonlocal isoperimetric problem. Communications in Mathematical Physics (to appear).

\bibitem{A} F.J. Almgren, Jr. Existence and regularity almost everywhere of solutions to elliptic variational problems with constraints, Memoirs AMS, 4, no. 165, 1976.

\bibitem{AMR} F. Andreu, J.M. Maz\'on, J.D. Rossi. The best constant for the Sobolev trace embedding from $W^{1,1}(\Omega)$ into $L^1(\partial\Omega)$. Nonlinear Anal. 59 (2004), 1125--1145. 

\bibitem{AB} M.S. Ashbaugh, R.D. Benguria.
Isoperimetric inequalities for eigenvalues of the Laplacian. Spectral theory and mathematical physics: a Festschrift in honor of Barry Simon's 60th birthday, 105--139, 
Proc. Sympos. Pure Math., 76, Part 1, Amer. Math. Soc., Providence, RI, 2007. 


\bibitem{B2} M. Bareket.
On an isoperimetric inequality for the first eigenvalue of a boundary value problem. 
SIAM J. Math. Anal. 8 (1977), 280--287. 



\bibitem{Bo2} M.H. Bossel. Membranes \'elastiquement li\'ees: inhomog\`enes ou sur une surface: une nouvelle extension du th\'eor\`eme isop\'erim\'etrique de Rayleigh-Faber-Krahn. Z. Angew. Math.
Phys. 39(5) (1988), 733--742.

\bibitem{BNT} B. Brandolini, C. Nitsch, C. Trombetti. An upper bound for nonlinear eigenvalues on convex domains by means of the isoperimetric deficit. Arch. Math. (Basel) 94 (2010), 391--400.

\bibitem{BDe} L. Brasco, G. De Philippis, B. Velichkov. Faber-Krahn inequalities in sharp quantitative form. preprint.

\bibitem{B} F. Brock. An isoperimetric inequality for eigenvalues of the Stekloff problem, Z. Angew.
Math. Mech. (ZAMM) 81 (2001), 69--71.

\bibitem{BD} F. Brock, D. Daners. Conjecture concerning a Faber-Krahn inequality for Robin problems.
In: Mini-Workshop: Shape Analysis for Eigenvalues
(Organized by: D. Bucur, G. Buttazzo, A. Henrot), Oberwolfach Rep. 4 (2007),
1022--1023.

\bibitem{BucD} D. Bucur, D. Daners.  An alternative approach to the Faber-Krahn inequality for Robin problems.
Calc. Var. Partial Differential Equations 37 (2010),  75--86.

\bibitem{Ci} A. Cianchi. A sharp trace inequality for functions of bounded variation in the ball. Proc. Roy. Soc. Edinburgh Sect. A 142 (2012), 1179--1191.

\bibitem{CFNT} A. Cianchi, V. Ferone, C. Nitsch, C. Trombetti.
Balls minimize trace constants in BV, preprint.

\bibitem{CL} M. Cicalese, G. P.  Leonardi. A selection principle for the sharp quantitative isoperimetric inequality. Arch. Ration. Mech. Anal. 206 (2012), 617--643.

\bibitem{D1} D. Daners.
A Faber-Krahn inequality for Robin problems in any space dimension.
Math. Ann.  333 (2006), 767--785.

\bibitem{D2} D. Daners.
Principal eigenvalues for generalized indefinite Robin problems. 
Potential Anal. 38 (2013), 1047--1069.

\bibitem{D3} D. Daners.
Krahn's proof of the Rayleigh conjecture revisited. 
Arch. Math. 96 (2011), 187--199.

\bibitem{DPF} M. del Pino, C. Flores. Asymptotic behavior of best constants and extremals for trace embeddings in expanding domains. Comm. Partial Differential Equations, 26 (2001), 2189--2210.

\bibitem{EFKNT} L. Esposito, V. Ferone, B. Kawohl, C. Nitsch,
C. Trombetti. The longest shortest fence and sharp Poincar\'e-Sobolev
inequalities. Arch. Rational. Mech. Anal. {206} (2012),
821--851.

\bibitem{ENT} L. Esposito, C. Nitsch, C. Trombetti. 
Best constants in Poincar\'e inequalities for convex domains. J. Convex Anal. 20 (2013), 253--264

\bibitem{Fa} G. Faber.  Beweis, dass unter allen homogenen Membranen von gleicher Fl\"ache und gleicher Spannung die kreisf\"ormige den tiefsten Grundton gibt.
M\"unch. Ber. (1923), 169--172.

\bibitem {FBLR}
J. Fern\'andez Bonder, E. Lami Dozo, J.D. Rossi. Symmetry properties for the extremals of the Sobolev trace embedding. Ann. Inst. H. Poincar\'e. Anal. Non Lin\'eaire 21 (6) (2004), 795--805.

\bibitem {FBR}
J. Fern\'andez Bonder, J.D. Rossi. Asymptotic behavior of the best Sobolev trace constant in expanding and contracting domains. Comm. Pure Appl. Anal. 1 (3) (2002), 359--378.

\bibitem {FBR1}
J. Fern\'andez Bonder, J.D. Rossi. On the existence of extremals for the Sobolev trace
embedding theorem with critical exponent. Bull. London Math. Soc. 37 (1) (2005), 119--125.

\bibitem{FNT} V. Ferone, C. Nitsch, C. Trombetti.
A remark on optimal weighted Poincar\'e inequalities for convex domains. Atti Accad. Naz. Lincei Cl. Sci. Fis. Mat. Natur. Rend. Lincei (9) Mat. Appl. 23 (2012), 467--475.

\bibitem{FM} A. Figalli, F. Maggi.
On the isoperimetric problem for radial log-convex densities.
Calc. Var. Partial Differential Equations 48 (2013), 447--489.

\bibitem{Fu1} B. Fuglede. Stability in the isoperimetric problem for convex or nearly spherical domains in $\R^n$,
Trans. Amer. Math. Soc., 314 (1989), 619--638.

\bibitem{FMP} N. Fusco, F. Maggi,  A. Pratelli. Stability estimates for certain Faber-Krahn, isocapacitary and Cheeger inequalities. Ann. Sc. Norm. Super. Pisa Cl. Sci. (5) 8 (2009), 51--71.

\bibitem{H} A. Henrot. {\it Extremum problems for eigenvalues of elliptic operators}. Frontiers in Mathematics, Birkh\"auser Verlag, Basel, 2006.

\bibitem{Ka} B. Kawohl, {\it Rearrangements and Convexity of Level Sets in PDE}, Lecture Notes in Math., vol. 1150, Springer-Verlag, Berlin, 1985.

\bibitem{K} S. Kesavan. {\it Symmetrization \& applications}. Series in Analysis, 3, World Scientific Publishing Co. Pte. Ltd., Hackensack, NJ, 2006.

\bibitem{Kr1} E. Krahn, \"Uber eine von Rayleigh formulierte Minimaleigenschaft des Kreises, Math. Ann.
94 (1925), 97--100.


\bibitem{Ma}
F. Maggi. {\it Sets of finite perimeter and geometric variational problems: an introduction to Geometric Measure Theory}. Cambridge Studies in Advanced Mathematics no. 135, Cambridge University Press, 2012.
 
\bibitem{MV} F. Maggi, C. Villani. Balls have the worst best Sobolev inequalities. J. Geom. Anal. 15 (2005), 83--121. 
 
\bibitem {MR}
S. Martinez, J.D. Rossi. Isolation and simplicity for the first eigenvalue of the p-laplacian with a nonlinear boundary condition. Abst. Appl. Anal., 7 (5) (2002), 287--293.

\bibitem{Maz} V. Maz'ya. \emph{Sobolev spaces with applications to elliptic partial differential equations}. Grundlehren der Mathematischen Wissenschaften, 342, Springer, Heidelberg, 2011. 

\bibitem {Mu}
\newblock C. M{\"u}ller. 
\newblock\emph{Spherical harmonics.}
\newblock Lecture Notes in Mathematics, 17, Springer-Verlag, Berlin-New York, 1966. 

\bibitem {N1} C. Nitsch. An isoperimetric result for the fundamental frequency via domain derivative.  Calc. Var. Partial Differential Equations 49 (2014), 323--335.

\bibitem {R1}
J.D. Rossi. First variations of the best Sobolev trace constant with respect to the domain.
Canad. Math. Bull. 51 (2008), 140--145. 

\bibitem{Sz} G. Szeg\"{o}, Inequalities for certain eigenvalues of a membrane of given area, J.
Rational Mech. Anal. {3} (1954), 343--356.

\bibitem{T1} I. Tamanini. Regularity results for almost minimal oriented hypersurfaces in $\R^N$, Quaderni del Dipartimento di Matematica dell'Universit\`a del Salento, 1, 1984; available for download at http://siba-ese.unile.it/index.php/quadmat

\bibitem{T2} I. Tamanini. Boundaries of Caccioppoli sets with H\"older-continuous normal vector. J. Reine Angew. Math. 334 (1982), 27--39.

\bibitem{Va} D. Valtorta. Sharp estimate on the first eigenvalue of the $p$-Laplacian on compact manifold with nonnegative Ricci curvature.
Nonlinear Anal. 75 (2012), 4974--4994.

\bibitem{We} H.F. Weinberger, An isoperimetric inequality for the $N$ dimensional free membrane
problem, J. Rational Mech. Anal. { 5} (1956), 633--636.

\bibitem{W} R. Weinstock, Inequalities for a classical eigenvalue problem, J. Rational Mech. Anal. 3 (1954), 745--753.



\end{thebibliography}
\end{document}